\documentclass[11pt]{article}
\usepackage{mathrsfs}
\usepackage{amsmath,amsthm}
\usepackage{amsfonts,amssymb,color}
\usepackage{cite}

\newtheorem{theorem}{Theorem}[section]
\newtheorem{lemma}[theorem]{Lemma}
\newtheorem{corollary}[theorem]{Corollary}

\theoremstyle{definition}

\newtheorem{remark}[theorem]{Remark}

\date{}

\def\bar{\overline}

\DeclareMathOperator{\diag}{diag}

\begin{document}
\title{\textbf{Signless Laplacian spectral analysis of a class of graph joins}\footnote{This
	research is supported the National Natural Science Foundation of China (No.~12361070) and the the Ministry
	of Science, Technological Development, and Innovation of the Republic of Serbia (No.~451-03-136/2025-03/200104).  Emails: yejiachang12@163.com (J. Ye), zstanic@matf.bg.ac.rs (Z. Stani\'c), jgqian@xmu.edu.cn (J. Qian).}}
\author{\small Jiachang Ye$^{1}$,
 Zoran Stani\' c$^{2}$, Jianguo Qian$^{1, 3}$ \\\small  $^1$ School of Mathematical Sciences, Xiamen University,  Xiamen, 361005, China\\\small $^{2}$ Faculty of Mathematics, University of Belgrade,
 Studentski trg 16, 11 000 Belgrade, Serbia
\\\small  $^3$ School of Mathematics and Statistics, Qinghai Minzu University,  Xining, 810007, China
} \maketitle

\begin{abstract} A graph is said to be determined by its signless Laplacian spectrum (abbreviated as DQS) if no other non-isomorphic graph shares the same signless Laplacian spectrum. In this paper, we establish the following results:
	\begin{itemize}
		\item[(1)] Every graph of the form $K_1 \vee (C_s \cup qK_2)$, where $q \ge 0$, $s \ge 3$, and the  number of vertices is at least $16$, is DQS;
		\item[(2)] Every graph of the form $K_1 \vee (C_{s_1} \cup C_{s_2} \cup \cdots \cup C_{s_t} \cup qK_2)$, where $t \ge 2$, $q \ge 0$, $s_i \ge 3$, and the  number of vertices is at least $52$, is DQS.
	\end{itemize}
	Here, $K_n$ and $C_n$ denote the complete graph and the cycle of order $n$, respectively, while $\cup$ and $\vee$ represent the disjoint union and the join of graphs. Moreover, the signless Laplacian spectrum of the graphs under consideration is computed explicitly.
	
\begin{flushleft}
\end{flushleft}
\end{abstract}

\section{Introduction}

Let $G = (V, E)$ be a finite simple undirected graph.  The number of vertices in $G$ is called the \emph{order} of $G$, denoted by $n(G)$ (or simply $n$). The number of edges is referred to as the \emph{size} of $G$, denoted by $m(G)$ (or $m$).

We use the notation $K_n$ and $C_n$ to denote, respectively, the complete graph and the cycle on $n$ vertices.

For two graphs $G$ and $H$, their \emph{disjoint union} is denoted by $G \cup H$. The disjoint union of $q$ copies of a graph $G$ is denoted by $qG$. Furthermore, the \emph{join} of $G$ and $H$, denoted by $G \vee H$, is the graph obtained from $G \cup H$ by adding all edges between every vertex of $G$ and every vertex of $H$.

Let $D(G)$ and $A(G)$ be the diagonal matrix of vertex degree sequence and the adjacency matrix of $G$, respectively.   The  \textit{signless Laplacian matrix} of $G$ is
$Q(G)=A(G)+D(G)$. Its eigenvalues are called  \textit{signless Laplacian eigenvalues}, and they form the \textit{signless Laplacian spectrum} of $G$. For brevity, the foregoing notions will be referred to as the \emph{$Q$-eigenvalues} and the \emph{$Q$-spectrum}, respectively. Moreover, when the context is clear, the prefix $Q$- will be omitted.

Two graphs are said to be \emph{$Q$-cospectral} if they share the same $Q$-spectrum. In this context, a graph $G$ is said to be \emph{determined by its signless Laplacian spectrum} (abbreviated as~\emph{DQS}) if no other non-isomorphic graph has the same $Q$-spectrum as $G$.

Identifying graphs that are, or are not, determined by the spectrum of a prescribed graph matrix is one of the oldest and most extensively studied problems in spectral graph theory. To the best of our knowledge, this line of inquiry dates back to the mid-1950s, when G\"{u}nthard and Primas first explored it in the context of chemical applications~\cite{GP}. Since then, the problem has evolved into a central topic in the  entire theory.

An experienced reader will readily acknowledge that determining whether a given graph is uniquely identified by its spectrum remains a challenging problem, even for graphs with seemingly simple structures. For foundational results and general developments, we refer the reader to~\cite{3,4}, while comprehensive treatments of DQS graphs can be found in~\cite{towI,towII}. Specific families of DQS graphs have been investigated in~\cite{CSS,Liu2,Wang-Friendship,Ye2024two,Ye2025DAM} and references therein. For a thorough survey and further discussion on spectral determination, see~\cite{18,towI,towII}.

In this paper we prove the following theorems.

 \begin{theorem}\label{12t} Every graph $K_1\vee (C_s\cup qK_2)$, with $q\ge 0$, $s\ge 3$, and at least $16$ vertices, is DQS.
\end{theorem}

 \begin{theorem}\label{13t} Every graph  $K_1\vee (C_{s_1}\cup C_{s_2}\cup\cdots \cup C_{s_t}\cup qK_2)$, with  $t\ge 2$, $q\ge 0$, $s_i\ge 3$ $(1\le i \le t)$, and at least $52$ vertices, is DQS.
\end{theorem}

It should be noted that the latter theorem does not generalize the former, due to the additional assumption on the order of the graph. The latter result can be regarded as a natural extension of those established in~\cite{20,Liu21}, which address the $Q$-spectral determination of joins between an isolated vertex and collections of vertex-disjoint cycles. Moreover, it generalizes the findings on the $Q$-spectral determination of the friendship graph presented in~\cite{Wang-Friendship}. Furthermore, Theorems~\ref{12t} and \ref{13t} are closely connected to the results in~\cite{MH3,Liu11,Liu1,sage,Ye2025DAM,Wang11,Zhou11}, which explore related graph products.

The proofs are carried out in a more general setting, wherein the $Q$-spectra of the graphs under consideration are explicitly computed. In particular, several auxiliary results are established in a form that also accommodates the case where the cycles appearing in the statements degenerate into two parallel edges.

The remainder of the paper is organized as follows. Section~\ref{sec:ir} introduces additional terminology and notation, as well as several known results. In Section~\ref{sec3}, we compute the $Q$-spectrum of the graphs under consideration and present auxiliary results concerning relationships among $Q$-eigenvalues. The proofs of Theorems~\ref{12t} and~\ref{13t} are provided in Sections~\ref{sec4} and~\ref{sec5}, respectively.

\section{Preliminaries}\label{sec:ir}

We write $N_{G}(v)$   and $d_{G}(v)$  to denote the set of neighbours of a vertex $v$
 and the degree of  the same vertex in a graph $G$, respectively. For a vertex subset $X\subset V$,  $G[X]$  denotes the subgraph induced by $X$. The graphs obtained by deleting edge $e$  and vertex $v$ of $G$ are denoted by $G-e$ and   $G-v$, respectively. The  $Q$-eigenvalues of graph $G$ of order $n$ are denoted by $$\sigma_{1}(G)\ge \sigma_{2}(G)\ge \cdots\ge \sigma_{n}(G).$$ Since $Q(G)$ is positive semidefinite, it holds $\sigma_n(G)\ge 0$. Moreover, the equality occurs if and only if $G$ has a bipartite component~\cite[Theorem~1.18]{Book-Stanic}. The $Q$-spectrum of $G$ is denoted by $S_Q(G)$; of course, it is considered as a multiset.

Henceforth, $\mathrm{mul}_{G}(\lambda)$ denotes the multiplicity of the eigenvalue~$\lambda$ in~$S_Q(G)$, while $\lambda^{(q)}$ denotes either $q$ copies of the real number $\lambda$, or a vector of length $q$ with all entries equal to $\lambda$, depending on the context.

 For a non-negative integer  $q$, the sum $$T_q(G)=\sum_{i=1}^{n}  \sigma_{i}^{q}(G)$$ is the $q$th  spectral moment of the signless Laplacian, or, equivalently, the $q$th $Q$-spectral moment of~$G$.

 We adopt a convention that  the degree $d_i(G)$  is attained by a vertex $v_i$ $(1\leq i\leq n)$ of $G$. In this context, $n_q(G)$  denotes the number of vertices of degree $q$ (for short, $q$-vertices) in  $V(G)\setminus \{v_1\}$, that is, $$n_q(G)=|\{u\,:\, u\in V(G)\setminus \{v_1\}~\text{and}~d_G(u)=q\}|,$$
where $v_1$ is a vertex attaining the maximum degree. Unless otherwise stated, the graph argument in the preceding notations will be omitted when no ambiguity arises.

In this section, as well as in Sections~\ref{sec4} and~\ref{sec5}, we assume that the vertex degrees of a graph are arranged as  $d_n\leq d_{n-1}\leq \cdots\leq d_1$.

The discussion continues with a selection of results from linear algebra, some of which are well known.
 They are drawn from \cite{BroSpe,Book-AGT,Book-Stanic}. Modifying the previous notation, let $\lambda_{i}(B)$, $1\le i\le n$, denote  the $i$th largest eigenvalue of an $n\times n$ real symmetric matrix~$B$.

  \begin{lemma}\label{lambda1-lem} Let $B$ be an $n \times n$ real symmetric matrix. Then its largest eigenvalue satisfies
  	\[
  	\lambda_1(B) = \max_{\| \boldsymbol{x} \| = 1} \boldsymbol{x}^\intercal B \boldsymbol{x},
  	\]
  	where $\boldsymbol{x} \in \mathbb{R}^n$ and $\| \cdot \|$ denotes the Euclidean norm.
\end{lemma}

\begin{lemma}\label{PFtheorem}  An irreducible non-negative $n\times n$  matrix $B$ has a real positive eigenvalue $\lambda_1$, such that the following statements hold:
\begin{itemize}
	\item[(i)] $|\lambda_i|\leq \lambda_1$ holds for all other eigenvalues
	$\lambda_i$, $2\leq i\leq n$;
	
	\item[(ii)] $\lambda_1$ is a simple root of the characteristic polynomial of $B$;
	
	\item[(iii)]  An eigenvector corresponding to $\lambda_1$  is all-positive.
\end{itemize}	
\end{lemma}

Suppose now that the columns of $B$ are indexed by
$X=\{1,2,\ldots,n\}$. For a partition ${X_1,X_2,\ldots,X_k}$, we set
\begin{equation*}
 B=   \begin{bmatrix}
    B_{1,1} & \ldots & B_{1,k}\\
    \vdots & \ddots & \vdots\\
    B_{k,1} & \ldots & B_{k,k}
    \end{bmatrix}
\end{equation*}
where $B_{i,j}$ denotes the block of $B$ formed by the rows in $X_i$ and the columns in $X_j$. If $q_{i,j}$ denotes the average row sum in $B_{i,j}$, then the matrix $N=[q_{i,j}]$ is a \textit{quotient matrix} of $B$. If, for every $i, j$, $B_{i,j}$  has a constant row sum, then the partition is called \textit{equitable}, and  $N$ is refined to \textit{equitable quotient matrix} of $B$.

\begin{lemma}\label{equitable}
Let $B$ be a non-negative irreducible real symmetric matrix, and $N$  an equitable quotient matrix of $B$. If $\lambda$ is an eigenvalue of $N$, then $\lambda$ is also an eigenvalue of  $B$. Moreover, the largest eigenvalues of $B$ and $N$ coincide.
 \end{lemma}

The remainder of this section is devoted to specific results concerning $Q$-eigenvalues. Only the statements are presented here; for further details, the reader is referred to the corresponding references.

\begin{lemma}\label{26l} {\rm\cite{Heu1}} If $G$ is a graph of order $n$ $(n\geq 3)$ and $e\in E(G)$, then
	$$\sigma_{1}(G)\ge \sigma_{1}(G-e)\ge \sigma_2(G)\ge \sigma_2(G-e)\ge \cdots\ge \sigma_{n}(G)\ge \sigma_{n}(G-e)\ge0.$$
\end{lemma}

\begin{lemma}\label{21l} {\rm\cite{Liu1}}  For a graph $G$, let $X=\{u_{1},u_{2},\ldots,u_k\} \subset V(G)$ and $H\cong G[X]$. If, for $1\leq j\leq k$, wee have $0\leq q_{j}\leq  d_G(u_{j}),$  then the inequality
	$$\lambda_{i} \big(\diag(q_{1},q_{2},\ldots,q_{k})+A(H)\big)\le  \sigma_{i}(G)$$ holds for  $1\leq i\leq k$.
\end{lemma}

\begin{lemma}\label{d1-d2-lem} {\rm\cite{Liu21}} If $G$ is a graph of order $n\geq 12$ satisfying $\sigma_{1}>n>5\ge \sigma_{2}\ge \sigma_{n}>0$, then $G$ is necessarily connected, along with $d_{1}\geq n-3$  and $d_{2}\leq4$.
\end{lemma}

\begin{lemma}\label{d1-lem} {\rm\cite{Ye2024two}} Let  $G$ be a  graph of order $n\ge 16$ satisfying $\sigma_{1}(G)>n>5> \sigma_{2}(G)\ge \sigma_{n}(G)>0$.   If  $H$ and $G$ are $Q$-cospectral, then $H$ is necessarily connected, with  $d_{2}(H)\leq 4$ and  $d_1(H)=d_1(G)\in \{n-1,n-2\}$.
\end{lemma}

\begin{lemma}\label{23l} {\rm\cite{Ye2024two}} Let $G$ be a connected  graph of order $n$, with $d_2\leq 4$. If either $d_1\ge 8>d_n\ge 2$  or $d_1\ge 11>1=d_n$, then $\sigma_1\le d_1+3$.
\end{lemma}

%\begin{lemma}\label{Pan-lem} {\rm\cite{YLpan}} If $G$ is a connected graph of order $n$  $(n\ge 2)$,  then \begin{align*}d_{1}+d_{2}\geq  \sigma_{1} \geq d_{1}+1,\end{align*}
%where the latter  equality holds   if and only if $G\cong K_{1,n-1}$.
%\end{lemma}

\begin{lemma}{\rm(cf.\cite{ZhangXD})}\label{kappa1-lem}  Let $G$ be a connected  graph of order $n\ge 2$. If ${\boldsymbol x}=(x(u_1),x(u_2),$ $\ldots,x(u_n))^\intercal$ is a unit real vector defined  on the vertex set $V(G)$,
	then  $$\sigma_1(G)\geq   \sum_{uv\in
		E(G)}(x(u)+x(v))^2,$$
	with equality  if and only if  ${\boldsymbol x}$ is   an eigenvector associated with $\sigma_1(G)$. \end{lemma}

Finally, we return to spectral moments. For a graph $G$, we write $\varsigma_G(F)$ to denote  the number of subgraphs isomorphic to $F$.

\begin{lemma}\label{TG-lem} {\rm\cite{18}} If $G$ is a  graph of order $n$ and size $m$,  then  $$  T_1(G)=\sum_{i=1}^{n}d_{i}=2m, \hspace{5pt}T_2(G)=\sum_{i=1}^{n}d^{2}_{i}+2m,\hspace{5pt}\text{and}\hspace{5pt}
T_3(G)=6\varsigma_G(C_3)+\sum_{i=1}^{n}d^{3}_{i}+3\sum_{i=1}^{n}d^{2}_{i}.$$
\end{lemma}

\section{$Q$-eigenvalues of certain joins}\label{sec3}

This section is concerned with the $Q$-eigenvalues of certain graph joins including, but not limited to, the following cases:
 $$K_1\vee (C_{s_1}\cup C_{s_2}\cup\cdots \cup C_{s_t}\cup qK_2),\quad K_1\vee (P_{3}\cup C_{s_1}\cup C_{s_2}\cup \cdots \cup C_{s_t}\cup qK_2) \quad \text{and} \quad K_1\vee (P_l\cup G_1),$$
where $P_l$ stands for the $l$-vertex path. We explicitly compute the $Q$-spectrum of the former graph.

\begin{lemma}\label{l3.1}Let $n$ be the order of $G\cong K_1\vee (C_{s_1}\cup C_{s_2}\cup\cdots \cup C_{s_t}\cup qK_2)$, with $t,q\ge 1$ and $s_i\ge 3$ $(1\le i\le t)$. Then
	$$S_Q(G)=\Big\{\lambda_1,\lambda_2,\lambda_3, 1^{(q)},3^{(q-1)},5^{(t-1)}, 3+2\cos\frac{2j\pi}{s_i}~:~ 1\le j\le s_i-1, 1\le i\le t\Big\},$$
	where $\lambda_1$, $\lambda_2$ and $\lambda_3$ are the  roots of $\lambda^3-(n+7)\lambda^2+(7n+8)\lambda-12n+4q+12$. Additionally, they satisfy $\lambda_1>n>5>\lambda_2>3>\lambda_3>1$.
\end{lemma}

\begin{proof} We suppose that $V(G)=\{u_i\,:\, 1\leq i\leq n\}$, where $d_G(u_n)=n-1$, $d_G(u_i)=2$ for $1\leq i\leq 2q$ and $u_{2j}u_{2j-1}\in E(G)$ for $1\leq j\leq q$. In what follows, we construct the eigenvectors for  certain eigenvalues.

We first deal with the eigenvalue 5. Let ${\boldsymbol \psi_j}=(\psi_{j}(u_1),\psi_{j}(u_2),\ldots,\psi_{j}(u_n))^\intercal$ for $1\leq j\leq t-1$, where $$\psi_{j}(u_i)=\left\{\begin{array}{rl}-s_{j+1},&\text{for}~u_i\in V(C_{s_1}),\\ s_1,& \text{for}~u_i\in V(C_{s_{j+1}}),\\ 0,&\text{otherwise}.\end{array}\right.$$

It is straightforward to verify that this collection forms a set of linearly independent eigenvectors corresponding to the eigenvalue $5$ of $Q(G)$. Consequently, $\mathrm{mul}_{G}(5) \geq t - 1$.

Secondly, we deal with the eigenvalues 1 and 3. As before, $${\boldsymbol \xi_1}=\big(1,-1,0^{(n-2)}\big)^\intercal,  {\boldsymbol \xi_2}=\big(0,0,1,-1,0^{(n-4)}\big)^\intercal, \ldots, {\boldsymbol \xi_q}=\big(0^{(2q-2)},1,-1,0^{(n-2q)}\big)^\intercal.$$
are linearly independent eigenvectors. Thus, $\mathrm{mul}_{G}(1)\ge q$. Similarly $${\boldsymbol \zeta_1}=(1,1,-1,-1,0^{(n-4)})^\intercal,  {\boldsymbol \zeta_2}=(0^{(2)},1,1,-1,-1,0^{(n-6)})^\intercal, \ldots, {\boldsymbol \zeta_{q-1}}=(0^{(2q-4)},1,1,-1,-1,0^{(n-2q)})^{\intercal}.$$
 are associated with 3,  when $q\ge 2$. Thus, $\mathrm{mul}_{G}(3)\ge q-1$.

For convenience, we denote $V(C_{s_i})=\{v_{i1},v_{i2},\ldots,v_{is_i}\}, 1\leq i\leq t$. Suppose that $A(C_{s_{i}}){\boldsymbol \omega}=\varrho {\boldsymbol \omega}$ (${\boldsymbol \omega}\ne {\bf 0}$), i.e., ${\boldsymbol \omega}$ is a unit eigenvector for $A(C_{s_{i}})$ corresponding  to the eigenvalue $\varrho$. It is well known that $\varrho=2\cos(2k\pi/s_i)$, $0\leq k\leq s_i-1$. Besides, it is easy to prove that if $\varrho\neq 2$  (i.e. $k\neq 0$), and then ${\boldsymbol \omega}=(\omega(v_{i1}),\omega(v_{i2}),\ldots,\omega(v_{is_{i}}))^{\intercal}$ satisfies $\sum_{j=1}^{s_{i}}\omega(v_{ij})=0$. Therefore, by excluding $\varrho=2$, we  construct an eigenvector $\bar{{\boldsymbol \omega}}=(\bar{\omega}(u_{1}), \bar{\omega}(u_{2}),\ldots,\bar{\omega}(u_{n}))^\intercal$  for $3+\varrho$ by setting
$$\bar{\omega}(u)=\left\{\begin{array}{rl}\omega(u),&\text{for}~u\in V(C_{s_{i}}),\\ 0,& \text{otherwise}.\end{array}\right.$$ To verify this, one may observe that $I_{s_{i}}+Q(C_{s_{i}})=3I_{s_{i}}+A(C_{s_{i}})$ is a principal submatrix of $Q(G)$. Therefore, $3+2\cos(2k\pi/s_i)$, $1\leq k\leq s_{i}-1$,  are the eigenvalues of $Q(G)$ for every $i~(1\leq i\leq t)$.

It remains to deal with the $Q$-eigenvalues denoted by $\lambda_1, \lambda_2$ and $\lambda_3$ in the statement formulation. The polynomial $\lambda^3 - (n+7)\lambda^2 + (7n+8)\lambda - 12n + 4q + 12$ is obtained as a factor of the minimal polynomial of $Q(G)$; it can also be derived from the characteristic polynomial of the corresponding quotient matrix.
 It is verified directly that the corresponding roots, $\lambda_1,\lambda_2$ and $\lambda_3$, satisfy $\lambda_1>n> 5>\lambda_2>3>\lambda_3>1$. Moreover, $${\boldsymbol \phi_i}=\big((\lambda_i-5)^{(2q)},(\lambda_i-3)^{(n-2q-1)},(\lambda_i-3)(\lambda_i-5)\big)^\intercal$$
is an eigenvector of $Q(G)$ corresponding to the eigenvalue $\lambda_i$, for $1\le i \le 3$.

An additional scenario must be considered: In certain cases, some of the computed eigenvalues may coincide, raising the question of whether their corresponding eigenvectors remain linearly independent.
From the already proved chain of, we obtain that $\lambda_1, \lambda_2$ and $\lambda_3$ are mutually distinct and do not belong to $\{1, 3, 5\}$. In addition,
$
1 \leq 3 + 2\cos\frac{2k\pi}{s_i} < 5,
$
for $1 \leq k \leq s_i - 1$ and $1 \leq i \leq t$.
In this context, if the equality
\[
3 + 2\cos\frac{2k\pi}{s_i} = \lambda_j
\]
holds for some choice of indices $i$, $j$ and $k$, then the corresponding eigenvectors are linearly independent by construction. The same conclusion applies if the left-hand side equals either $1$ or $3$. This completes the argument and the entire proof.
\end{proof}

\begin{remark}\label{r3.2} Note that, in the previous lemma,  $3+2\cos(2j\pi/s_i)=1$ occurs if and only if $s_i$ is even and $j=s_i/2$. Thus, $\mathrm{mul}_G(1)=q+c_{e}(C)$, where $c_{e}(C)$ denotes the number of even cycles among~$C_{s_i}$. Besides, we also have $$n<\lambda_1=\sigma_1(G)<n+2,~ \sigma_n(G)=1<\lambda_3,~\sigma_2(G)<5~ (\text{when}~ t=1),~\text{and}~\sigma_2(G)=5~(\text{when}~t\ge 2).$$ Last but not least, for any fixed order $n$, the largest eigenvalue $\sigma_1(G)$ depends on $q$ but is independent of both the number of cycles and their lengths. This conclusion follows from the preceding proof and will also be corroborated by Corollary~\ref{coro3.4} below.
\end{remark}

We now extend the context to include the possibility of cycles $C_2$, which are regarded as digons consisting of two parallel edges connecting the same pair of vertices. In this setting, the degree of a vertex is defined as the number of edges incident to it, and the entries of the adjacency matrix represent the multiplicity of the corresponding edges.

\begin{lemma}\label{l3.3}Let $G^*\cong K_1\vee (C_{s_1}\cup C_{s_2}\cup\cdots \cup C_{s_t}\cup qK_2)$, $t,q\ge 1$ and $s_i\ge 2$ $(1\le i\le t)$. If $n$ is the order of $G^*$, then $\sigma_1(G^*)$
is equal to the largest root of $\lambda^3-(n+7)\lambda^2+(7n+8)\lambda-12n+4q+12$.
\end{lemma}
\begin{proof} The characteristic polynomial of the equitable quotient matrix
\begin{equation*}
 N=   \begin{bmatrix}
    n-1 & n-1-2q & 2q\\
    1 & 5 & 0\\
    1 & 0 & 3
    \end{bmatrix}.
\end{equation*}
is $\lambda^3-(n+7)\lambda^2+(7n+8)\lambda-12n+4q+12.$
By Lemma \ref{equitable},  $\sigma_1(G^*)=\lambda_1(N)$.
\end{proof}

Combining this result with Lemma \ref{l3.1}, Remark \ref{r3.2} and Lemma~\ref{l3.3}, we  immediately obtain an interesting  corollary.
\begin{corollary}\label{coro3.4} Given a graph $G\cong K_1\vee (C_{s_1}\cup C_{s_2}\cup\cdots \cup C_{s_t}\cup qK_2)$ and a multigraph $G^*\cong K_1\vee (C_{s'_1}\cup C_{s'_2}\cup\cdots \cup C_{s'_r}\cup qK_2)$, where $t,r,q\ge 1$, $s_i\ge 3$ $(1\le i\le t)$ and $s'_i\ge 2$ $(1\le i\le r)$. If $n(G)=n(G^*)$, then
$\sigma_1(G)=\sigma_1(G^*)$.
\end{corollary}

We now turn to the second join product, continuing within the broader context.

\begin{lemma}\label{P3-lem} Let  $G^*\cong K_1\vee (P_{3}\cup C_{s_1}\cup C_{s_1}\cup \cdots \cup C_{s_t}\cup qK_2)$, where $t, q\ge 0$ and $s_i\ge 2$ $(1\le i \le t)$.Then  the least $Q$-eigenvalue $\sigma_n(G^*)$ is less than $1$.
\end{lemma}
\begin{proof} For $t,q \ge 1$, the characteristic polynomial of the equitable quotient matrix
\begin{equation*}
 N=   \begin{bmatrix}
    n-1 & n-2q-4 &2 &1 & 2q\\
    1 & 5 & 0&0&0\\
    1 & 0 & 2&1&0\\
    1&0&2&3&0\\
    1&0&0&0&3
    \end{bmatrix}.
\end{equation*}is $$g(\lambda)=\lambda^5-(n+12)\lambda^4+(12n + 47)\lambda^3+(4q - 51n - 52)\lambda^2+(88n - 20q - 48)\lambda+16q - 48n + 72.$$
From $g(1)=8>0$ and
$g(0)=16q - 48n + 72<-32n+72<0$,
we deduce the existence of an eigenvalue of $N$ lying in the interval $(0,\,1)$. By Lemma \ref{equitable},  $\sigma_n(G^*)<1$.

Cases $q=0$ or $t=0$ simplify the previous computation, and lead to the same conclusion.
\end{proof}

In what follows, we will use the following setting. Let \begin{equation}\label{eq:G}G\cong K_1\vee (P_l\cup G_1)\end{equation}
be a graph with   $V(G)=\{u_1,u_2,\ldots,u_n\}$, where  $d_G(u_n)=n-1$. The vertices of $P_l$ are denoted by $u_1, u_2,\ldots, u_l$, in the natural order. Also, $G_1$ denotes for any graph.  Let   ${\boldsymbol \alpha}=(\alpha_1,\alpha_2,\ldots,\alpha_n)^{\intercal}$ be the Perron eigenvector (positive unit eigenvector associated with the largest eigenvalue) of $G$ such that $\alpha_i$  corresponds to vertex $u_i$, $1\leq i\leq n$.

\begin{lemma}{\rm\cite{Ye2024two}}\label{special-lem}  Under the above notation, the following statements hold:
\begin{itemize}
\item[(i)] If  $l\geq 2$, then  $\alpha_i=\alpha_{l+1-i}$ holds for  $1\leq i\leq \left\lfloor\frac{l}{2}\right\rfloor$;
\item[(ii)] If  $l\geq 6$, $\alpha_i \neq \alpha_{i-2}$ holds for $3\le i\le \lfloor\frac{l}{2}\rfloor$;
\item[(iii)] If  $l\geq 4$,  $\alpha_i \neq \alpha_{i+1}$  holds for  $1\le i\le \lfloor\frac{l}{2}\rfloor-1$;
\item[(iv)] If $l \ge 3$, then $\alpha_1<\alpha_2$.
\end{itemize}
\end{lemma}

Here is an extension.

\begin{lemma}\label{PathCycle-lem}In addition to the previous lemma:
\begin{itemize}
\item[(i)] If $l\ge 5$, then
$$\sigma_1(G)<\sigma_1(K_{1}\vee (C_{s}\cup P_{l-s}\cup G_1))$$ holds for $3\leq s\leq l-2;$
\item[(ii)] If $l=4$, then
$$\sigma_1(G)<\sigma_1(K_{1}\vee (C_{2}\cup P_{2}\cup G_1)),$$
where $C_2$ is the digon.
\end{itemize}
\end{lemma}

\begin{proof} We first prove (ii). Denote $G^*\cong K_{1}\vee (C_{2}\cup P_{2}\cup G_1)$. It is easy to see that $G^*\cong G-u_1u_2-u_3u_4+u_1u_4+u_2u_3$, where $G$ is given by \eqref{eq:G}. Now, invoking Lemmas \ref{lambda1-lem}, \ref{kappa1-lem} and \ref{special-lem}, we obtain
$$\sigma_1(G^*)-\sigma_1(G)\geq {\boldsymbol \alpha}^\intercal Q(G^*){\boldsymbol \alpha}-{\boldsymbol
	\alpha}^\intercal Q(G){\boldsymbol \alpha}=2(\alpha_{3}-\alpha_{1})(\alpha_{2}-\alpha_{4})=2(\alpha_{2}-\alpha_{1})^2>0,$$
as desired.

(i): Suppose that $\widetilde{G}\cong K_{1}\vee (C_{s}\cup P_{l-s}\cup G_1)$. In this case, we restrict our attention to the subcase where $l$ is even, as the alternative  can be proved by analogous arguments.

Suppose $l=2a$ with $a\ge 3$. By Lemma~\ref{special-lem}, we have
$\alpha_{a-j}=\alpha_{a+j+1}$,  $0\leq j\leq a-1$.

If
$s=2k+1$ ($k\ge 1$), then $a\geq k+2$ follows from $l-s\geq
2$. Obviously, $$\widetilde{G}\cong
G-u_{a-k+1}u_{a-k}-u_{a+k+1}u_{a+k+2}+u_{a-k+1}u_{a+k+1}+u_{a-k}u_{a+k+2}.$$
Note that $\alpha_{a-k}=\alpha_{a+1+k}$, $\alpha_{a+1-k}=\alpha_{a+k}$ and $\alpha_{a-k-1}=\alpha_{a+k+2}$. So, by Lemma \ref{kappa1-lem},
   \begin{eqnarray*}
\sigma_1(\widetilde{G})-\sigma_1(G)&\geq &{\boldsymbol \alpha}^\intercal Q(\widetilde{G}){\boldsymbol \alpha}-{\boldsymbol \alpha}^\intercal Q(G){\boldsymbol \alpha}    \nonumber\\
&=&2\alpha_{a+1+k}\alpha_{a-k+1}+2\alpha_{a+k+2}\alpha_{a-k}-2\alpha_{a-k}\alpha_{a-k+1}-2\alpha_{a+k+1}\alpha_{a+k+2}   \nonumber\\
&=&2\alpha_{a-k}\alpha_{a-k+1}+2\alpha_{a-k-1}\alpha_{a-k}-2\alpha_{a-k}\alpha_{a-k+1}-2\alpha_{a-k}\alpha_{a-k-1}\nonumber\\
&=&0.\nonumber
\end{eqnarray*}
If the equality holds, then
${\boldsymbol \alpha}$ is  the Perron eigenvector of $\sigma_1(\widetilde{G})$.  Moreover,
$$3\alpha_{a-k}+\alpha_{a-k-1}+\alpha_{a+k+2}=3\alpha_{a-k}+\alpha_{a-k-1}+\alpha_{a-k+1},$$ which yields
$\alpha_{a-k+1}=\alpha_{a+k+2}=\alpha_{a-k-1}$. However, this contradicts the statement of Lemma~\ref{special-lem}(ii).

On the other hand, if $s=2k$ $(k\ge 2)$, then $a\geq
k+1$ deduces from $l-s\geq 2$. Moreover, $$\widetilde{G}\cong
G-u_{a-k+1}u_{a-k}-u_{a+k}u_{a+k+1}+u_{a-k+1}u_{a+k}+u_{a-k}u_{a+k+1}.$$
Combining with Lemma~\ref{special-lem}, we arrive at
    \begin{eqnarray*}
\sigma_1(\widetilde{G})-\sigma_1(G)&\geq &{\boldsymbol \alpha}^\intercal Q(\widetilde{G}){\boldsymbol \alpha}-{\boldsymbol \alpha}^\intercal Q(G){\boldsymbol \alpha}    \nonumber\\
&=&2\alpha_{a+1-k}\alpha_{a+k}+2\alpha_{a-k}\alpha_{a+k+1}-2\alpha_{a-k}\alpha_{a-k+1}-2\alpha_{a+k}\alpha_{a+k+1}   \nonumber\\
&=&2\alpha_{a+1-k}^{2}+2\alpha_{a-k}^{2}-4\alpha_{a-k}\alpha_{a-k+1} \nonumber\\
&=&2(\alpha_{a+1-k}-\alpha_{a-k})^{2}\nonumber\\
&>&0,\nonumber
\end{eqnarray*}
 because $\alpha_{a+1-k}\neq \alpha_{a-k}$, and the proof is completed.
\end{proof}

\begin{remark}\label{PathCycle-r} It is straightforward to verify that if $G_1$ is a multigraph, then Lemmas~\ref{special-lem} and~\ref{PathCycle-lem} remain valid, since the matrices $Q(G)$, $Q(\widetilde{G})$, and $Q(G^*)$ continue to be non-negative, irreducible, real, and symmetric.
\end{remark}

\section{Proof  of Theorem~\ref{12t}}\label{sec4}

Let $G$ be as in the formulation of Theorem~\ref{12t}. A particular case is resolved in \cite{20}: For $q=0$, the wheel graph $G$ is DQS. Therefore, we suppose that $q\ge 1$. Let further $H$ be a simple graph  that is $Q$-cospectral with~$G$. We set  $n(H)=n(G)=n$.

Based on Lemma~\ref{l3.1} and Remark~\ref{r3.2}, the following ordering of eigenvalues is deduced:
\[
1 = \sigma_n(H) < 3 < \sigma_2(H) < 5 < n < \sigma_1(H).
\]

For brevity, denote $n_i(G)$, $d_j(G)$, $n_i(H)$, and $d_j(H)$ by $n_i$, $d_j$, $\widetilde{n}_i$, and $\widetilde{d}_j$, respectively, where $0 \leq i \leq n-1$ and $1 \leq j \leq n$.

The following result is established first.

\begin{lemma}\label{l4.1} If $n\ge 16$, then $G$ and $H$ share the same vertex degrees and $\varsigma_H(C_3)=\varsigma_G(C_3)$.
\end{lemma}
\begin{proof} From Lemma~\ref{d1-lem}, we know that $H$ is connected with $\widetilde{d}_1=n-1$ and $\widetilde{d}_2\le 4$ (this inequality will be frequently used) when $n\ge 16$. By inserting $n_3=s$, $n_2=2q$, $n_1=n_4=0$  and  $d_1=\widetilde{d}_1=n-1$ in Lemma~\ref{TG-lem}, we arrive at	
	\begin{equation}\label{e4.1}\left\{
		\begin{array}{ll}  \widetilde{n}_1+\widetilde{n}_2+\widetilde{n}_3+\widetilde{n}_4=2q+s,\\
            \widetilde{n}_1+2\widetilde{n}_2+3\widetilde{n}_3+4\widetilde{n}_4=4q+3s,\\
			\widetilde{n}_1+4\widetilde{n}_2+9\widetilde{n}_3+16\widetilde{n}_4=8q+9s,\\
            \widetilde{n}_1+8\widetilde{n}_2+27\widetilde{n}_3+64\widetilde{n}_4+6\varsigma_H(C_3)=16q+27s+6\varsigma_G(C_3).
		\end{array}
		\right.\end{equation}

From the first two equalities, we have
\begin{equation}\label{e4.2}
-\widetilde{n}_1+\widetilde{n}_3+2\widetilde{n}_4=s.
\end{equation}
Besides, from the first and the  third equality in \eqref{e4.1}, we deduce
\begin{equation}\label{e4.3}
-3\widetilde{n}_1+5\widetilde{n}_3+12\widetilde{n}_4=5s.
\end{equation}
Now combining equalities \eqref{e4.2} and \eqref{e4.3}, we immediately obtain
\begin{equation*}\label{e4.4}
\widetilde{n}_1+\widetilde{n}_4=0.
\end{equation*}
Further by \eqref{e4.2}, we have $\widetilde{n}_3=s$. Moreover, from the first equality in  \eqref{e4.1}, we obtain $\widetilde{n}_2=2q$. Therefore, $G$ and $H$  share vertex degrees. Now, from the last equality in  \eqref{e4.1}, we arrive at $\varsigma_H(C_3)=\varsigma_G(C_3)$. Obviously, if $s\ge 4$, then $\varsigma_H(C_3)=\varsigma_G(C_3)=s+q$. Otherwise, that is, for $s=3$, we have $\varsigma_H(C_3)=\varsigma_G(C_3)=s+1+q=q+4$.
\end{proof}

Now we eliminate all but two possibilities for $H$.

\begin{lemma}\label{l4.2} For $n\geq 16$, the only structural possibilities for $H$ are
	$$K_1\vee (C_{k}\cup P_{l_1}\cup P_{l_2}\cup\cdots \cup P_{l_q}) \quad \text{and} \quad K_1\vee (P_{l_1}\cup P_{l_2}\cup\cdots \cup P_{l_q}),$$ where $3\le k\le s$ and $l_1\ge l_2\ge \cdots \ge l_r \ge 3>l_{r+1}=\cdots =l_q=2$, with $0 \leq r \leq q$ in the former case and $1 \leq r \leq q$ in the latter.
\end{lemma}
\begin{proof} We denote the vertices of $H$ by $v_1, v_2, \ldots, v_n$, and adopt the convention for vertex degrees elaborated in Section~\ref{sec:ir}. By Lemma \ref{l4.1}, $\widetilde{d}_1=n-1$, so we can suppose $d_H(v_1)=n-1$. Besides, $\widetilde{d}_2=3$ and $\widetilde{n}_1=0$. These degree conditions lead to the conclusion that every component of  $H-v_1$ is isomorphic to either $C_k~(k\ge 3)$ or $P_l~(l\ge 2)$. Moreover, from $\sigma_2(H)<5$, $H$ cannot contain $K_1\vee (C_{k_1}\cup C_{k_2})$ as a subgraph; otherwise, Lemmas \ref{26l} and \ref{21l} imply $$\sigma_2(H)\ge \lambda_2(I_{k_1+k_2}+Q(C_{k_1}\cup C_{k_2}))=5.$$ Hence, at most one component of  $H-v_1$ is a cycle. Moreover, $\widetilde{n}_2=2q$ implies that there are exactly $q$ disjoint paths with length at least 2 in  $H-v_1$. Therefore, the lemma holds because $H$ and $G$ share degrees.
\end{proof}

It remains to eliminate the remaining two possibilities, listed in the formulation of the previous lemma. For convenience, we denote $$\widetilde{H}\cong K_1\vee (C_{k}\cup P_{l_1}\cup P_{l_2}\cup\cdots \cup P_{l_q}) \quad \text{and} \quad H^*\cong K_1\vee (P_{l_1}\cup P_{l_2}\cup\cdots \cup P_{l_q}).$$

\medskip\noindent{\em \textbf{Proof of  Theorem \ref{12t}.}} According to the discussion in the beginning of this section, we suppose that $q\geq 1$. We also set $n\geq 16$ (as in the statement of this theorem). We shall prove that the  graph $H$, introduced at the beginning of this section, is isomorphic to $G$. For a contradiction, we suppose $H\not \cong G$. Then by Lemma~\ref{l4.2}, $H \cong \widetilde{H}$ or $H \cong H^*$.
\vspace*{0.1cm}\\
\smallskip\noindent{\textit{Case 1: $H \cong \widetilde{H}$.}}
By Lemma~\ref{l4.2},
we have  $3\le k<s$ and $l_1\ge l_2\ge \cdots \ge l_r \ge 3>l_{r+1}=\cdots =l_q=2$ $(1\le r \le q)$.

\smallskip\noindent{\textit{Subcase 1.1: $l_r \ge 4$.}}
 By Corollary~\ref{coro3.4}, Lemma~\ref{PathCycle-lem} and Remark~\ref{PathCycle-r}, it follows that
  $$\sigma_1(\widetilde{H})<\sigma_1\big(K_1\vee (C_{k}\cup C_{l_1-2}\cup C_{l_2-2}\cup \cdots \cup C_{l_r-2}\cup qK_2)\big)=\sigma_1(G),$$ which is a contradiction.

\smallskip\noindent{\textit{Subcase 1.2: $l_r=3$.}}
Lemmas \ref{26l} and \ref{P3-lem} imply $$\sigma_n(\widetilde{H})\le \sigma_n\big(K_1\vee (P_3\cup C_{k}\cup C_{l_1}\cup \cdots C_{l_{r-1}}\cup (q-r)K_2)\big)<1=\sigma_n(G),$$
violating the assumption on $Q$-cospectrality.

\smallskip\noindent{\textit{Case 2: $H \cong H^*$.}}
 Again, Lemma~\ref{l4.2} gives $l_1\ge l_2\ge \cdots \ge l_r \ge 3>l_{r+1}=\cdots =l_q=2$ $(1\le r \le q)$.
As before, two subcases arise: $l_r \geq 4$ and $l_r = 3$. In the former situation, a similar argument yields
\[
\sigma_1(H^*) < \sigma_1\big(K_1 \vee (C_{l_1 - 2} \cup C_{l_2 - 2} \cup \cdots \cup C_{l_r - 2} \cup qK_2)\big) = \sigma_1(G).
\]
In the latter one, it follows that
\[
\sigma_n(H^*) \leq \sigma_n\big(K_1 \vee (P_3 \cup C_{l_1} \cup \cdots \cup C_{l_{r-1}} \cup (q - r)K_2)\big) < 1 = \sigma_n(G).
\]
Both conclusions contradict the $Q$-cospectrality assumption, completing the proof.
\qed

\section{Proof  of Theorem~\ref{13t}}\label{sec5}

Let $G$ be as in the formulation of Theorem~\ref{13t}. On the basis of \cite{Liu21}, which resolves a particular case, we may suppose that $q\ge 1$ and $t\ge 2$. Let further $H$ be a simple graph  that is $Q$-cospectral with~$G$.

Also by Lemma~\ref{l3.1} and Remark \ref{r3.2}, we deduce the following setting: $$1=\sigma_n(H)<\sigma_2(H)=5<n<\sigma_1(H).$$

Here we adopt the same notation for $n_i(G)$, $d_j(G)$, $n_i(H)$ and $d_j(H)$ as the previous section.  From Lemma~\ref{d1-d2-lem}, we know that when $n\ge 12$, $H$ is connected with $\widetilde{d}_{2}\le 4$.

The first lemma tells us about the multiplicity of $1$ in the $Q$-spectrum of a graph containing vertices that share a particular neighbourhood.

\begin{lemma}{\rm\cite{Ye2025DAM}}\label{eigenvalue1-lem} Let $F$ be a  graph of order $n~(n\ge 2)$ such that $N_F(u_1)=N_F(u_2)=\cdots =N_F(u_s)=\{u_{s+1}\}$,  where $u_i\in V(F)$ $(1\leq i\leq s+1)$ and $s\ge 1$. Then $\mathrm{mul}_{F}(1)\ge s-1$.
\end{lemma}

We proceed by proving the following result.

\begin{lemma}\label{l5.1} If $n\ge 52$ or $q\ge 12$, then $\widetilde{d}_{1}=d_1=n-1$.
\end{lemma}
\begin{proof} By Lemma~\ref{d1-d2-lem}, we have $\widetilde{d}_{1}\ge n-3\ge 11$ when $n\ge 14$. By Lemma \ref{23l}, we have $n<\sigma_1(H)\le \widetilde{d}_{1}+3$, and thus $\widetilde{d}_{1}\ge n-2$. For a contradiction, we suppose that $\widetilde{d}_{1}=n-2$.
	
	Lemma \ref{TG-lem}, leads to the system:
	\begin{equation}\label{e51}\left\{
		\begin{array}{ll}
2\widetilde{n}_{1}+2\widetilde{n}_{4}=2n_1+2n_4+d_1(d_1-5)-\widetilde{d}_{1}(\widetilde{d}_{1}-5),\\
\widetilde{n}_{2}-3\widetilde{n}_{4}=n_2-3n_4-d_1(d_1-4)+\widetilde{d}_{1}(\widetilde{d}_{1}-4),\\
2\widetilde{n}_{3}+6\widetilde{n}_{4}=2n_3+6n_4+d_1(d_1-3)-\widetilde{d}_{1}(\widetilde{d}_{1}-3).
		\end{array}
		\right.\end{equation}
	Since $d_1=n-1$, $\widetilde{d}_{1}=n-2$, $n_2=2q$, $n_3=n-1-2q$ and $n_1=n_4=0$, from \eqref{e51}, we have
	\begin{equation}\label{e52}\left\{
		\begin{array}{ll}  \widetilde{n}_{1}=n-4-\widetilde{n}_{4}\\
			               \widetilde{n}_{2}=2q-2n+7+3\widetilde{n}_{4},\\
                            \widetilde{n}_{3}=2n-2q-4-3\widetilde{n}_{4}.
		\end{array}
		\right.\end{equation}
	Now, from \eqref{e52}, we find
	\begin{align*}\label{e53}\widetilde{n}_{2}+\widetilde{n}_{3}=3.
	\end{align*}
	Therefore, only four possible scenarios remain, each of which is considered separately.

	\smallskip\noindent{\textit{Case 1: $\widetilde{n}_{2}=0$ and $\widetilde{n}_{3}=3$.}}  Remark~\ref{r3.2} implies $\mathrm{mul}_G(1)\le q+(n-1-2q)/4=(n+2q-1)/4$, while  by \eqref{e52}, we get $\widetilde{n}_{4}=(2n-2q-7)/3$ and $\widetilde{n}_{1}=(n+2q-5)/3$. Note that $\widetilde{d}_1=n-2$ and $n\ge 2q+7$, since $t\ge 2$. If $n\ge 40$ (or $q\ge 9$),  Lemma~\ref{eigenvalue1-lem} leads to the impossible scenario: $\mathrm{mul}_H(1)\ge \widetilde{n}_{1}-2>(n+2q-1)/4\ge \mathrm{mul}_G(1)$.

\smallskip\noindent{\textit{Case 2: $\widetilde{n}_{2}=1$ and $\widetilde{n}_{3}=2$.}}
Here, \eqref{e52} leads to $\widetilde{n}_{4}=(2n-2q-6)/3$ and $\widetilde{n}_{1}=(n+2q-6)/3$, whereas Lemma~\ref{eigenvalue1-lem} implies $\mathrm{mul}_H(1)\ge \widetilde{n}_{1}-2>(n+2q-1)/4\ge \mathrm{mul}_G(1)$, whenever $n\ge 44$ (or $q\ge 10$).

\smallskip\noindent{\textit{Case 3: $\widetilde{n}_{2}=2$ and $\widetilde{n}_{3}=1$.}}
From $\widetilde{n}_{4}=(2n-2q-5)/3$ and $\widetilde{n}_{1}=(n+2q-7)/3$, we arrive at $\mathrm{mul}_H(1)\ge \widetilde{n}_{1}-2>(n+2q-1)/4\ge \mathrm{mul}_G(1)$ for  $n\ge 48$ (or $q\ge 11$).

\smallskip\noindent{\textit{Case 4: $\widetilde{n}_{2}=3$ and $\widetilde{n}_{3}=0$.}}
Remark~\ref{r3.2} gives $\mathrm{mul}_G(1)\le (n+2q-1)/4$, and~\eqref{e52} gives  $\widetilde{n}_{4}=(2n-2q-4)/3$ and $\widetilde{n}_{1}=(n+2q-8)/3$.  Lemma~\ref{eigenvalue1-lem} eliminates the possibility $n\ge 52$ (or $q\ge 12$).

Therefore, if $n\ge 52$ or $q\ge 12$, then $\widetilde{d}_1=d_1$.
\end{proof}

The next two lemmas are formulated based on straightforward logical reasoning.

\begin{lemma}\label{l5.2} If $n\ge 52$ or $q\ge 12$, then $G$ and $H$ share the same vertex degrees.
\end{lemma}
\begin{proof} Lemma~\ref{l5.1} ensures $\widetilde{d}_1=n-1$, whenever $n\ge 52$ or $q\ge 12$. By inserting $d_1=\widetilde{d}_1$, $n_2=2q$, $n_3=n-1-2q$ and  $n_1=n_4=0$  in~\eqref{e51}, we arrive at
$$\begin{aligned}\label{32e}  \widetilde{n}_1=\widetilde{n}_4=0, \,\,  \widetilde{n}_3=n-1-2q \,\, ~\text{and}\,\, \widetilde{n}_2=2q. \end{aligned}$$
Therefore, $G$ and $H$ are as desired.
\end{proof}

\begin{lemma}\label{l4.2A} For $n \geq 52$, the only structural possibilities for $H$ are:
	\[
	K_1 \vee (C_{s'_1} \cup C_{s'_2} \cup \cdots \cup C_{s'_a} \cup P_{l_1} \cup P_{l_2} \cup \cdots \cup P_{l_q}) \quad \text{and} \quad K_1 \vee (P_{l_1} \cup P_{l_2} \cup \cdots \cup P_{l_q}),
	\]
	where $s'_i \geq 3$, $l_1 \geq l_2 \geq \cdots \geq l_r \geq 3 > l_{r+1} = \cdots = l_q = 2$, with $0 \leq r \leq q$ in the former case and $1 \leq r \leq q$ in the latter.
\end{lemma}
\begin{proof} Lemma~\ref{l5.1} ensures $\widetilde{d}_1=n-1$ when $n\ge 52$, and we denote the corresponding vertex by $v_1$. In addition, Lemma~\ref{l5.2} implies $\widetilde{d}_2=3$ and $\widetilde{n}_1=0$. These degree conditions lead to the conclusion that every component of  $H-v_1$ is isomorphic to either $C_s~(s\ge 3)$ or $P_l~(l\ge 2)$.  Moreover, $\widetilde{n}_2=2q$ implies that there are exactly $q$ disjoint paths of length at least 2 in  $H-v_1$. Therefore, the statement of the lemma holds because $H$ and $G$  share degrees.
\end{proof}

As above, we denote $$\widetilde{H}\cong K_1\vee (C_{s'_1}\cup C_{s'_2}\cup\cdots \cup C_{s'_a}\cup P_{l_1}\cup P_{l_2}\cup\cdots \cup P_{l_q}) \quad \text{and} \quad H^*\cong K_1\vee (P_{l_1}\cup P_{l_2}\cup\cdots \cup P_{l_q}).$$

\medskip\noindent{\em \textbf{Proof of  Theorem \ref{13t}.}} We suppose that $q\geq 1$ and $n\geq 52$. Let $H$ be as in the beginning of this section. Suppose that $H\not \cong G$. Then by Lemma~\ref{l4.2A}, $H \cong \widetilde{H}$ or $H \cong H^*$.
\vspace*{0.1cm}\\
\smallskip\noindent{\textit{Case 1: $H \cong \widetilde{H}$.}}
By Lemma~\ref{l4.2A}, we have $s'_i\ge 3$ and $l_1\ge l_2\ge \cdots \ge l_r \ge 3>l_{r+1}=\cdots =l_q=2$ $(0\le r \le q)$.

\smallskip\noindent{\textit{Subcase 1.1: $r=0$.}}
 Then $$H \cong \widetilde{H} \cong K_1\vee (C_{s'_1}\cup C_{s'_2}\cup\cdots \cup C_{s'_a}\cup qK_2).$$ Firstly, $a=t$ since $\mathrm{mul}_H(5)=a-1=t-1=\mathrm{mul}_G(5)$ (by Lemma~\ref{l3.1}). Besides, we suppose $s_{1}\ge s_{2} \ge\cdots \ge s_t\ge 3$, $s'_{1}\ge s'_{2} \ge\cdots \ge s'_t\ge 3$. Note that $\max\{3+2\cos\frac{2j\pi}{s_1}\,:\:1\leq j\leq s_1-1\}=3+2\cos\frac{2\pi}{s_1}$, we obtain $s_1=s'_1$ immediately. By excluding the $Q$-eigenvalues $3+2\cos\frac{2j\pi}{s_1},$ $1\leq j\leq s_1-1$, from the common $Q$-spectrum given in Lemma~\ref{l3.1}, we have $s_2=s'_2$. Similarly, we can get $s_i=s'_i$, for all $i$, which implies $H \cong G$.

\smallskip\noindent{\textit{Subcase 1.2: $r\ge 1$ and $l_r \ge 4$.}}
 By employing Corollary \ref{coro3.4}, Lemma \ref{PathCycle-lem} and Remark \ref{PathCycle-r}, we obtain  $$\sigma_1(\widetilde{H})<\sigma_1\big(K_1\vee (C_{s'_1}\cup C_{s'_2}\cup\cdots \cup C_{s'_a}\cup C_{l_1-2}\cup C_{l_2-2}\cup \cdots \cup C_{l_r-2}\cup qK_2)\big)=\sigma_1(G),$$ which contradicts the $Q$-cospectrality.

\smallskip\noindent{\textit{Subcase 1.3: $r\ge 1$ and  $l_r=3$.}}
 By Lemmas \ref{26l} and \ref{P3-lem}, $$\sigma_n(\widetilde{H})\le \sigma_n\big(K_1\vee (P_3\cup C_{s'_1}\cup C_{s'_2}\cup\cdots \cup C_{s'_a}\cup C_{l_1}\cup \cdots C_{l_{r-1}}\cup (q-r)K_2)\big)<1=\sigma_n(G),$$
contrary the same assumption.

\smallskip\noindent{\textit{Case 2: $H \cong H^*$.}}
By Lemma~\ref{l4.2A}, we have $l_1\ge l_2\ge \cdots \ge l_r \ge 3>l_{r+1}=\cdots =l_q=2$ $(1\le r \le q)$, and there are two subcases that correspond to Subcases 1.2 and 1.3 of the previous part.

For $l_r \ge 4$, the same arguments lead to the impossible scenario $\sigma_1(H^*)<\sigma_1\big(K_1\vee (C_{l_1-2}\cup C_{l_2-2}\cup \cdots \cup C_{l_r-2}\cup qK_2)\big)=\sigma_1(G).$

Similarly, $l_r=3$  reveals $\sigma_n(H^*)\le \sigma_n\big(K_1\vee (P_3\cup C_{l_1}\cup \cdots C_{l_{r-1}}\cup (q-r)K_2)\big)<1=\sigma_n(G),$
which is impossible.

The proof is completed.
\qed

%\section*{Acknowledgements}
%The authors would like to   thank the anonymous referee  for his valuable suggestions and comments which lead to an improvement of the original manuscript.


\begin{thebibliography}{99}
	\bibitem{BroSpe}A.E. Brouwer, W.H. Haemers,  Spectra of Graphs, Springer, New York, 2012.
	

	\bibitem{18}D. Cvetkovi\'c, P. Rowlinson, S.K. Simi\'c, Signless Laplacians of finite graphs, Linear Algebra Appl., 423 (2007), 155--171.
	
	\bibitem{towI} D. Cvetkovi\' c, S.K. Simi\' c, Towards a spectral theory of graphs based on signless Laplacian, I, Publ. Inst. Math. (Belgrade), 85 (2009), 19--33.
	
	\bibitem{towII} D. Cvetkovi\' c, S.K. Simi\' c, Towards a spectral theory of graphs based on signless Laplacian, II, Linear Algebra Appl., 432 (2010), 1708--1733.
	
	
	\bibitem{CSS} D. Cvetkovi\' c, S.K. Simi\' c, Z. Stani\' c, Spectral determination of graphs whose components are paths and cycles, Comput. Math. Appl., 59 (2010), 3849--3857.
	
	
	
	\bibitem{3} E.R. van Dam,  W.H. Haemers,  Which graphs are determined by their
	spectrum?, Linear Algebra Appl.,  373 (2003), 241--272.
	
	\bibitem{4}E.R. van Dam,  W.H. Haemers,  Developments on spectral characterizations of graphs, Discrete Math.,  309 (2009), 576--586.
	
	\bibitem{Book-AGT}C. Godsil, G. Royle, Algebraic Graph Theory, Springer-Verlag, New York, 2001.
	
	\bibitem{GP} H.H. G\"unthard, H. Primas, Zusammenhang von Graphtheorie und Mo-Theotie von Molekeln mit Systemen konjugierter Bindungen,  Helv. Chim. Acta, 39 (1956), 1645--1653.
	
	
	
	\bibitem{Heu1} J.V.D. Heuvel, Hamilton cycles and eigenvalues of
	graphs,  Linear Algebra Appl., 226-228 (1995), 723--730.
	
	
	
	\bibitem{20}  M.  Liu, Some graphs determined by their  (signless) Laplacian   spectra,  Czechoslovak Math. J.,  62 (2012),  1117--1134.

\bibitem{Liu2}  M. Liu,  X. Gu, H. Shan, Z. Stani\'c, Spectral characterization of the complete graph removing a cycle,  J. Combin. Theory  A, 205 (2024), 105868.


	
	\bibitem{MH3}  M. Liu,  Y. Yuan,  K.C. Das, The fan graph is determined by its signless Laplacian spectrum,   Czechoslovak
	Math. J., 70   (2020), 21--31.



	\bibitem{Liu1}M. Liu, Y. Zhu, H. Shan, K.C. Das, The spectral characterization of butterfly-like graphs, Linear Algebra  Appl., 513 (2017), 55--68.
	
	\bibitem{Liu11} X. Liu, S. Wang, Laplacian spectral characterization of some graph products, Linear Algebra Appl.,
437 (2012), 1749--1759.


	

	
	\bibitem{sage} X. Liu, Y. Zhang, X. Gui, The multi-fan graphs are determined by their Laplacian spectra, Discrete Math.,  308 (2008), 4267--4271.
	
%	 \bibitem{YLpan}  Y.-L. Pan,  Sharp upper bounds for the Laplacian graph eigenvalues,   Linear Algebra Appl.,  355 (2002), 287--295.
	
	\bibitem{Book-Stanic} Z. Stani\' c, Inequalities for Graph Eigenvalues, Cambridge University Press, Cambridge, 2015.

\bibitem{Wang-Friendship} J. Wang, F. Belardo, Q. Huang, B. Borovi\'canin, On the two largest $Q$-eigenvalues of graphs, Discrete Math., 310 (2010), 2858-2866.


	\bibitem{Wang11}J. Wang, H. Zhao, Q. Huang,  Spectral characterization of multicone graphs,  Czechoslovak
	Math. J.,  62 (2012), 117--126.
	
	
%	\bibitem{WangWei}  W. Wang, A simple arithmetic criterion for graphs being determined by their generalized spectra, J. Combin. Theory B, 122 (2017), 438--451.
	
	\bibitem{Ye2024two} J. Ye, M. Liu,  Z. Stani\'c, Two classes of graphs determined by the signless Laplacian spectrum, Linear Algebra Appl., 708 (2025), 159--172.

\bibitem{Ye2025DAM} J. Ye, J. Qian,  Z. Stani\'c, Determining some graph joins by the signless Laplacian spectrum, Discrete Applied Math., 375 (2025), 17--24.
	
	\bibitem{Liu21}  G. Zhang, M. Liu, H. Shan, Which $Q$-cospectral graphs have same degree sequences, Linear Algebra Appl., 520 (2017), 274--285.
	
\bibitem{ZhangXD} X.-D. Zhang, The Laplacian spectral radii of trees with degree sequences,  Discrete Math., 308 (2008), 3143--3150.	

	\bibitem{9} Y. Zhang, X. Liu, X.  Yong, Which wheel graphs are determined by their Laplacian spectra?, Comput. Math. Appl., 58 (2009),
	1887--1890.
	
	\bibitem{Zhou11} J. Zhou, C. Bu, Laplacian spectral characterization of some graphs obtained by product operation,
	Discrete Math.,  312 (2012), 1591--1595.
	
	
\end{thebibliography}
\end{document}